\documentclass{amsart}
\usepackage[utf8]{inputenc}
\usepackage{amsmath, amssymb, amsthm, amsfonts, mathtools, array, xfrac, float, multirow, indentfirst, url, enumerate, comment, afterpage, pifont, makecell, xcolor, hyperref}
\usepackage[backend=bibtex, sorting=none, bibstyle=alphabetic, citestyle=alphabetic, sorting=nyt,maxbibnames=99, giveninits=true, isbn=false, url=false, doi=false, eprint=false]{biblatex}
\usepackage[a4paper, total={6in, 8in}]{geometry}

\newtheorem{theorem}{Theorem}[section]
\newtheorem{lemma}[theorem]{Lemma}

\newtheorem{corollary}[theorem]{Corollary}
\newtheorem{example}{Example}
\newtheorem{remark}{Remark}[section]
\renewcommand{\Re}{\textrm{Re}}

\theoremstyle{definition}

\newcommand{\C}[0]{\mathbb{C}}
\newcommand{\R}[0]{\mathbb{R}}
\newcommand{\Z}[0]{\mathbb{Z}}
\newcommand{\N}[0]{\mathbb{N}}
\newcommand{\cS}[0]{\mathcal{S}}
\newcommand{\md}[0]{\mathrm{d}}

\newcolumntype{M}[1]{>{\centering\arraybackslash}m{#1}}
\newcolumntype{N}{@{}m{0pt}@{}}

\numberwithin{equation}{section}
\begin{document}

\title{On Turing's method for a general set of $L$-functions and the Selberg class}
\author[N.\ Paloj\"arvi, \and T.\ Zhao]{Neea Paloj\"arvi and Tianyu Zhao}

\address{
    NP: School of Science, The University of New South Wales (Canberra), ACT 2600, Australia.
}
\email{n.palojarvi@unsw.edu.au}
\address{
    TZ: Department of Mathematics, The Ohio State University, 231 West 18th
    Ave, Columbus, OH 43210, USA.
}
\email{zhao.3709@buckeyemail.osu.edu}

\subjclass[2020]{Primary: 11M41, 11Y35, 11L07.}
\keywords{$L$-functions, Selberg class, explicit bounds, Turing's method}

\begin{abstract}
We derive explicit bounds for two general classes of $L$-functions, improving and generalizing earlier known estimates. These bounds can be used, for example, to apply Turing's method for determining the number of zeros up to a given height. 
\end{abstract}

\maketitle

\section{Introduction}

Estimating the sizes of $L$-functions is a fundamental problem in analytic number theory, especially due to its intimate connection to the study of their zeros. For example, for the Riemann zeta-function $\zeta(s)$, the prototypical example of an $L$-function, classical methods for finding zeros have long been known. In 1953, Turing \cite{Turing1953} proposed an efficient method that provides an upper bound for the number of zeros in a specified range, which can thus be used to confirm the completeness of a list of computed zeros in the range, or to partially verify the truth of the Riemann hypothesis. The central ingredient of this method is a quantitative bound on a definite integral of the argument of $\zeta(s)$, which further boils down to bounding $|\zeta(s)|$ on the critical line. See Trudgian's work \cite{Trudgian2011} for refinements and extensions of the method to Dirichlet $L$-functions and Dedekind zeta-functions. 

The present article builds on \cite{Booker2006} where Booker developed Turing's method for a general class of $L$-functions, including entire Artin $L$-functions, and used it partially verify Artin's conjecture in certain cases. We aim to accomplish two things. We first refine the various estimates in \cite{Booker2006}, which ultimately leads to an improved bound on the integral of the argument, a key component of Turing's method. A potentially useful result proved along the way is an explicit convexity bound for more general $L$-functions that resembles, for instance, Radamacher's estimate \cite[Theorems 3 and 4]{Rademacher1959} for Dirichlet $L$-functions and Dedekind zeta-functions. This improves a number of results in the existing literature, as we shall briefly discuss in Section \ref{sec:examples}. Second, we present the analogue of Turing's method for $L$-functions in the Selberg class, highlighting the similarities and variations compared to the first case.

We begin with the definitions and axioms for the two classes of $L$-functions we are concerned with.

\subsection{Selberg class}
The Selberg class of functions $\cS$ was introduced in~\cite{SelbergOldAndNew} and consists of Dirichlet series
\begin{equation}
\label{eq:DS}
L(s) = \sum_{n=1}^{\infty} \frac{a(n)}{n^s}, \quad \Re(s)>1
\end{equation}
where $s=\sigma+i  t$ for real numbers $\sigma$ and $t$, satisfying the following axioms:
\begin{enumerate}[(i)]
    \item \emph{Ramanujan hypothesis}. We have $a(n)\ll_{\varepsilon} n^{\varepsilon}$ for any $\varepsilon>0$.
    \item \emph{Analytic continuation}. There exists $k=k_L\in\N_{0}$ such that $(s-1)^{k_L}L(s)$ is an entire function of finite order.
    \item \emph{Functional equation}. The function $L(s)$ satisfies 
    \begin{equation}
    \label{eq:SelbergFunctionEq}
   \mathcal{L}(s)=\omega\overline{\mathcal{L}(1-\bar{s})}, \text{ where } \mathcal{L}(s)=L(s)N^s\prod_{j=1}^{f}\Gamma\left(\lambda_j s+\mu_j\right)=:\gamma_L(s)L(s)
    \end{equation}
    with $\left(N,\lambda_j\right)\in\R_{+}^2$ and $\left(\mu_j,\omega\right)\in\C^2$ where $\Re(\mu_j)\geq 0$ and $|\omega|=1$. 
    \item \emph{Euler product}. The function $L(s)$ has a product representation
    \[
    L(s) = \prod_{p}\exp{\left(\sum_{l=1}^{\infty}\frac{b\left(p^l\right)}{p^{ls}}\right)}, \quad \Re(s)>1
    \]
    where the coefficients $b\left(p^l\right)$ satisfy $b\left(p^l\right)\le C_L  p^{l\theta_{L}}$ for some constants $C_L$ and $0\leq\theta_{L}<1/2$ for all $p,l$.
\end{enumerate}
We note that from the assumptions it follows that $a(1)=1$. 

Examples of functions in $\cS$ include the Riemann zeta-function $\zeta(s)$ and Dirichlet $L$-functions $L(s,\chi)$ for a primitive characters $\chi$ modulo $q$. As seen from the definitions, elements in $\cS$ are assumed to satisfy similar types of properties as $\zeta(s)$. For example, the \textit{Generalized Riemann hypothesis} claims that for all zeros $\rho$ in the strip $0 \leq \Re(s) \leq 1$ that do not arise from the poles of the gamma-factors, we have $\Re(\rho)=1/2$.

In some parts of this article, we can benefit from the fact that we replace axiom (iv) with a more detailed axiom called \textit{the polynomial Euler product}. This means that we assume that
    \[
    L(s) = \prod_{p}\prod_{j=1}^{l}\left(1-\frac{\alpha_j(p)}{p^s}\right)^{-1}, \quad \Re(s)>1,
    \]
    where $l\in\N$ is the \emph{order of a polynomial Euler product} and $\alpha_j(p)\in\C$, $1\leq j\leq l$, are some functions which are defined for every prime number $p$. We also denote 
\begin{equation}
    \label{def:lambda}
    \lambda:=\prod_{j=1}^f \lambda_j^{2\lambda_j}.
\end{equation}
The term $\lambda$ is invariant, meaning that for a fixed Selberg class function $F$, the value of the term $\lambda$ is always the same for even different versions of the functional equation \eqref{eq:SelbergFunctionEq} (see e.g. \cite{odzak2011}).

    
    A Selberg class function may have several possible representations in the form \eqref{eq:SelbergFunctionEq}. In some of the results below we assume that $\lambda_j<1$ for all $j$. In the next lemma we explain that this is a reasonable assumption since for all Selberg class functions we can find a version of the functional equation \eqref{eq:SelbergFunctionEq} that satisfies the wanted property.
    
    \begin{lemma}
    \label{lemma:Sf}
        Assume that $L$ belongs to the Selberg class. Then it has a functional equation of type \eqref{eq:SelbergFunctionEq} such that for all $j$ we have $\lambda_j<1$.
    \end{lemma}

\subsection{A general class}
\label{sec:Artin}

For convenience we record the assumptions in \cite[Section 1.3]{Booker2006}.
\begin{itemize}
    \item $L(s)$ has a polynomial Euler product of order $r$, as in (v) above:
    \[
    L(s) = \prod_{p}\prod_{j=1}^{r}\left(1-\frac{\alpha_j(p)}{p^s}\right)^{-1}, \quad \Re(s)>1,
    \]
    where $|\alpha_j(p)|\leq p^{\theta}$ for some $\theta<1/2$, and we assume the product is absolutely convergent. Moreover, for each $p$ there is a permutation $(\alpha'_j(p))_{j=1}^r$ of $(\alpha_j(p))_{j=1}^r$ such that $|\alpha_j(p)\alpha'_j(p)| \leq 1$.

    \item Define 
    \begin{align}
    \label{def:Lambdas}
        \Lambda(s):=\gamma(s)L(s), \quad
        \gamma(s):=\omega N^{(s-1/2)/2}\prod_{j=1}^r \left(\pi^{-s/2}\Gamma\left(\frac{s+\mu_j}{2}\right)\right)
    \end{align}
    where $N \in \mathbb{Z}_+$, $\Re(\mu_j)\geq -\theta$ for some $\theta<1/2$ and $|\omega|=1$ such that the functional equation $\Lambda(s)=\overline{\Lambda(1-\overline{s})}$ is satisfied.

    \item Let $Q$ be the analytic conductor:
    \begin{equation}
    \label{eq:defQ}
    Q(s):=N\prod_{j=1}^r \frac{s+\mu_j}{2\pi}.
    \end{equation}
    Also define $\chi(s):=\frac{\overline{\gamma(1-\overline{s})}}{\gamma(s)}$ such that $L(s)=\chi(s)\overline{L(1-\overline{s})}$.
    \item $L(s)$ may have finitely many poles, all lying on the 1-line. Label them by $1+\tau_k$ for $k=1,\ldots, m$. From the functional equation we see that each $\tau_k$ coincides with $-\mu_j$ for some $j$. Setting
    \[
    P(s):=\prod_{k=1}^m (s-\tau_k),
    \]
    then $\Lambda(s)P(s)P(s-1)$ is entire.
\end{itemize}

 We remark that the Selberg class gives more freedom for the functional equation and product representation (general vs. polynomial one) whereas the assumptions in this section provide more freedom for the location of poles. It is conjectured that the Selberg class functions satisfy the strong-$\lambda$ hypothesis, meaning that for all $L\in \mathcal{S}$ we can find a functional equation where $\lambda_j=1/2$ for all $j$, and that all of them have a polynomial Euler product. If both of these conjectures are true, the Selberg class would be a subset of functions defined in this section. However, since these conjectures remain unproven, the Selberg class may or may not be a subset of the set described in this section and hence it is meaningful to study the Selberg class separately.

\subsection{Turing's method}
\label{sec:Turing}
Turing's method provides a way to determine the zeros of an $L$-function assuming that they are simple. It is based on the idea that simple zeros in the half-line are located between sign changes of the function $\Lambda(1/2+it)$. Here $\Lambda(s)$ is as in Section \ref{sec:Artin} if we consider a general set of functions defined in Section \ref{sec:Artin} and $\Lambda_{L}(s):=z\mathcal{L}(s)$ if we are considering the Selberg class functions. Here $z$ is a certain complex number with $|z|=1$ that is given in Section \ref{sec:rigorous}. Note that $\Lambda_{L}(s)=\overline{\Lambda_{L}(1-\overline{s})}$, and thus $\Lambda_{L}(1/2+it) \in \mathbb{R}$. 

We now have a closer look at the method. If $t$ is not the ordinate of a zero or pole of $\Lambda$, let us define
\begin{equation}
\label{eq:defS}
    S(t):=-\frac{1}{\pi}\Im\left(\int_{1/2}^\infty \frac{L'}{L}(\sigma+it)\, d\sigma\right),
\end{equation}
otherwise let $S(t)=\lim_{\varepsilon \to 0^+} S(t+\varepsilon)$. Moreover, for $t_1<t_2$, let $N(t_1, t_2)$ denote the number of zeros of a function $L$ (counting multiplicity) with imaginary part in $(t_1, t_2]$. If neither $t_1$ nor $t_2$ is the ordinate of the zero or pole, then, as in \cite[Equation (4--1)]{Booker2006}, we can deduce that
\begin{equation*}
    N(t_1, t_2)=\frac{1}{2\pi}\Im\int_C \frac{\Lambda'}{\Lambda}(s) \, ds=\frac{1}{\pi}\Im \int_{C \cap H} \frac{\gamma'}{\gamma}(s) \, ds+\frac{1}{\pi}\Im \int_{C \cap H} \frac{L'}{L}(s) \, ds,
\end{equation*}
where $C$ is the counterclockwise oriented rectangle with corners at $2+it_1$, $2+it_2$, $-1+it_2$ and $-1+it_1$, and $H$ is the half-plane $\{s \in \mathbb{C}: \Re(s) \geq 1/2\}$. The definition of $\gamma(s)$ depends on whether we are applying the method for $L$-functions described in Section \ref{sec:Artin} or the Selberg class functions.
Hence, we have
\begin{equation*}
    N(t_1, t_2)=\frac{1}{\pi} \left[\log{\gamma(s)}\right]_{1/2+it_1}^{1/2+it_2}+S(t_2)-S(t_1).
\end{equation*}
Let us choose the branch for $\gamma(s)$ by using the principal branch of $\log{\Gamma}$, and set
\begin{equation*}
    \Phi(t):=\frac{1}{\pi}\left(\arg{\omega}+\frac{t\log{N}}{2}-\frac{\log{\pi}}{2}\left(rt+\Im\left(\sum_{j=1}^r \mu_j\right)\right)+\Im\left(\sum_{j=1}^r \log{\Gamma\left(\frac{1/2+it+\mu_j}{2}\right)}\right)\right)
\end{equation*}
for functions in the general set and
\begin{equation*}
    \Phi(t):=\frac{1}{\pi}\left(t\log{N}+\Im\left(\sum_{j=1}^f \log{\Gamma\left(\frac{\lambda_j}{2}+i\lambda_jt+\mu_j\right)}\right)\right)
\end{equation*}
for the Selberg class functions. Let $N(t):=\Phi(t)+S(t)$, so that $N(t_1, t_2)=N(t_2)-N(t_1)$.

From the sign changes of $\Lambda(1/2+it)$ we can find simple zeros between $t_1$ and $t_2$ on the line $\Re(s)=1/2$. Namely, if $\Lambda(1/2+ia)$ and $\Lambda(1/2+ib)$ have different signs for soma real numbers $a<b$, there must be a zero in the interval $(a,b)$. If it turns out that all of the zeros having imaginary parts between $t_1$ and $t_2$ lie on the $\frac{1}{2}$-line and are simple, then we can find all such zeros. To verify that we have found all of the zeros with imaginary parts in $(t_1,t_2]$, we remember that $S(t)=N(t)-\Phi(t)$ has a mean value $0$. Hence, the graph of $N(t_0,t)-\Phi(t)$, for a fixed real number $t_0$, oscillates around a constant value. Comparing this graph to the graph presenting the found number of zeros with imaginary parts in $(t_0,t]$ minus $\Phi(t)$, we can find if we have missed any zeros. Namely, for any missed zero there should be an obvious difference between these two graphs.

Hence, in order to apply Turing's method, we would like to estimate the function $\int_{t_1}^{t_2} S(t)\,\md t$ well in order to have an accurate graph to be compared with the found number of zeros. In addition, we want to present a fast way to compute $\Lambda(1/2+it)$. An improved estimate for $\int_{t_1}^{t_2} S(t)\,\md t$ is provided in Theorem \ref{thm:S} for a general set of functions and in Theorem \ref{thm:SelbergL} for certain functions in $\cS$ (including most, if not all, of the known cases of interest). A fast way to determine the function $\Lambda(1/2+it)$ in the case of the general set of functions given in Section \ref{sec:Artin} was already provided by Booker in \cite[Section 5]{Booker2006}. We revise this method for a large subset of $\cS$ in Section \ref{sec:rigorous}.


\section{Main results}
\label{sec:results}

\subsection{Results}
Our main developments for the previous research are:
\begin{itemize}
    \item using Stirling's formula to obtain better estimates for $|L(s)|$ in $[-\varepsilon, 1+\varepsilon]$ (see Lemma \ref{lemma:L}). Here, we have chosen $\varepsilon \in (0,1/2]$ instead of $\varepsilon =1/2$, which gives more freedom to optimize the parameters;
    \item deriving estimates for $|L(s)|$ for the Selberg class functions;
    \item the previous two points lead to (improved) estimates for the integral of the function $S(t)$ defined in \eqref{eq:defS}; see our main theorems below;
    \item revising computations involved in Turing's method for the Selberg class (under some additional assumptions; see Section \ref{sec:rigorous}).
\end{itemize}

\begin{theorem}
\label{thm:S}
Assume that $L$ belongs to the general set of functions defined in Section \ref{sec:Artin}. Let $X>5$. Assume $\varepsilon \in (\theta, 1/2]$, $t_2 \geq t_1$, $\left|t_2+\Im(\mu_j)\right|\geq 3/2$ and $(t_1+\Im(\mu_j))^2 \ge \left(2+\varepsilon+\Re(\mu_j)\right)^2+X^2$ for all $j=1,...,r$. Then
    \begin{multline*}
        \pi \int_{t_1}^{t_2} S(t)\,\md t \leq \left(\frac{1}{16}+\frac{\varepsilon(1+\varepsilon)}{4}\right) \log |Q(3+\varepsilon+it_2)|+\frac{A_\varepsilon-1}{2}\log |Q(1+\varepsilon+it_1)| \\
        +rc_\theta(\varepsilon)+\frac{r\left(8.3+0.09A_\varepsilon+(1/2+\varepsilon)(2+\varepsilon)/5\right)}{X},
    \end{multline*}
    where $Q(s)$ is as in \eqref{eq:defQ}, 
     \begin{equation}
    \label{def:Aepsilon}
        A_\varepsilon:=\left(\frac{1}{2}+\varepsilon\right)\log{\left(1+\frac{1}{1/2+\varepsilon}\right)}+\log{\left(\frac{3}{2}+\varepsilon\right)},
    \end{equation}
    and
    \begin{equation}
        \begin{split}
        \label{eq:cdef}
            z_\theta(\sigma):=& \left(\frac{\zeta(2\sigma+2\theta)\zeta(2\sigma-2\theta)}{\zeta(\sigma+\theta)\zeta(\sigma-\theta)}\right)^{1/2},\\
            Z_\theta(\sigma):=&\left(\zeta(\sigma+\theta)\zeta(\sigma-\theta)\right)^{1/2},\\
            c_\theta(\varepsilon):=&\left(\frac{1}{2}+\varepsilon\right)\log Z_\theta(1+\varepsilon)+\int_{1+\varepsilon}^\infty \log Z_\theta(\sigma)\, \md \sigma-\int_{1+\varepsilon}^{2+\varepsilon} \log z_\theta(\sigma)\, \md \sigma, \\
            &\hspace{2cm}-\int_{3/2}^\infty \log z_\theta(\sigma)\, \md \sigma+A_\varepsilon \frac{z_\theta'}{z_\theta}(1+\varepsilon).
        \end{split}
    \end{equation}
\end{theorem}

We now state a version of Theorem~\ref{thm:S} specialized for the Selberg class. In this context we put $S(t)$ in \eqref{eq:defS} to be $S_L(t)$. Note that according to Lemma \ref{lemma:Sf} we can choose a representation where $\lambda_j<1$ for all $j$, and hence the following lemma can be applied for all Selberg class functions with $L(1) \neq 0$.

\begin{theorem}
\label{thm:SelbergL}
Let $L$ be a Selberg class function and and choose a version of the functional equation that we have $\lambda_j<1$ for all $j$. Let
$$
\varepsilon \in \left[0, \min\left\{\frac12,\frac{1}{2}(\max_j\{\lambda_j\}^{-1}-1)\right\}\right]
$$
such that $\varepsilon>\theta_{L}$. Let $X>5$. Assume $t_2\geq t_1\geq 0$, $t_2\geq \varepsilon$ and $(t_1+\Im(\mu_j))^2 \ge \left(\lambda_j(2+\varepsilon)+\Re(\mu_j)\right)^2+X^2$ for all $j=1,...,f$. Suppose also that $L(1)\neq 0$. Then
    \begin{multline}
    \label{eq:SEstimate}
        \pi \int_{t_1}^{t_2} S_L(t)\,\md t \leq \left(\frac{1}{16}+\frac{\varepsilon(1+\varepsilon)}{4}\right) \log |Q_{L}(1+\varepsilon+it_2)|+\frac{A_\varepsilon-1}{2}\log |Q_{L}(1+\varepsilon+it_1)| \\
        +c_{\theta_{L}}(\varepsilon)+\frac{f}{X}\left(0.9+\frac{4A_\varepsilon}{5}\right)+\frac{k_L(2.5+\varepsilon)}{\max\{\varepsilon,t_1\}},
    \end{multline}
    where $A_\varepsilon$ is as in \eqref{def:Aepsilon}, 
    \begin{equation}
    \label{def:Qs}
        Q_{L}(s):=N^{2}\prod_{j=1}^f \left(\lambda_j s+\mu_j+1\right)^{2\lambda_j},
    \end{equation}
    \begin{equation*}
    Z_{\theta_{L}}(\sigma):=\exp\left(\sum_p\sum_{l=1}^\infty\frac{C_L}{p^{(\sigma-\theta_{L})l}}\right),
\end{equation*}
\begin{equation*}
    z_{\theta_{L}}(\sigma):=\exp\left(-\sum_p\sum_{l=1}^\infty\frac{C_L}{p^{(\sigma-\theta_{L})l}}\right)
\end{equation*}
    and
    \begin{multline*}
        c_{\theta_{L}}(\varepsilon):=\left(\frac{1}{2}+\varepsilon\right)\log Z_{\theta_{L}}(1+\varepsilon)+\int_{1+\varepsilon}^\infty \log Z_{\theta_{L}}(\sigma)\, \md \sigma-\int_{1+\varepsilon}^{2+\varepsilon} \log z_{\theta_{L}}(\sigma)\, \md \sigma \\
        -\int_{3/2}^\infty \log z_{\theta_{L}}(\sigma)\, \md \sigma+A_\varepsilon \frac{z_{\theta_{L}}'}{z_{\theta_{L}}}(1+\varepsilon).
    \end{multline*}
\end{theorem}

\begin{theorem}
\label{thm:SelbergPolynomial}
If we have otherwise the same conditions as in Theorem \ref{thm:SelbergL} but  additionally assume that $L$ has a polynomial Euler product representation of order $l$, then $\theta_{L}=0$, and we can replace the term $c_{\theta_{L}}(\varepsilon)$ in \eqref{eq:SEstimate} with $lc_0(\varepsilon)$. Here $c_{0}(\varepsilon)$ is given as in \eqref{eq:cdef}.
\end{theorem}

\subsection{Examples of the results and improvements to earlier estimates}
\label{sec:examples}

Below we give some examples illustrating how our bounds on $L$-functions improve some earlier results. The proofs are short and will be given in Section \ref{sec: proofs of examples}.

\begin{example}
\label{ex:Artin}
Let $L$ be an entire Artin $L$-function. Then
\begin{equation}
\label{eq:LArtin}
        \left|L(s)\right| \leq \zeta(1.49)^rN^{\frac{1.49-\sigma}{2}}\left(\frac{|3+s|}{2\pi}\right)^{\frac{(1.49-\sigma)r}{2}}
    \end{equation}
     for $\Re(s) \in[0.5,1.49]$.
\end{example}
Since $N \geq 3$ in \cite[Theorem 3.2]{PM2011}, the right-hand side of \eqref{eq:LArtin} improves \cite[Lemma 5]{GL2022} for all $\Re(s) \in[0.5,1.49]$ and $|t| \geq 140$. Moreover, for all $\varepsilon \in(0, 1/2)$ and $|t|$ large enough, Lemma \ref{lemma:L} improves Lemma 5 in \cite{GL2022}. Our improvements to \cite[Lemma 5]{GL2022} improve the constant $C_2$ in \cite[Lemma 7]{GL2022} when $\sigma<3/2$, which in turn improves \cite[Lemma 9]{GL2022} where Lemma 7 is used to bound $\frac{L'(1/2+\delta+it,\chi)}{L(1/2+\delta+it,\chi)}$. It also affects the constant $C_6$ in \cite[Lemma 10]{GL2022}, which is defined in terms of $C_2$. 

\begin{example}
\label{ex:newform}
 Let $f \in L^2(\Gamma_1(N)\setminus \mathbb{H})$ be a cuspidal Maass newform of weight $0$ and level $N$, and $L(s)$ an $L$-function associated with $f$. Then, for $\sigma \in [-0.4, 1.4]$
   \begin{equation}
    \label{eq:Maass}
        |L(s)| < \zeta\left(1.4+\frac{7}{64}\right)\zeta\left(1.4-\frac{7}{64}\right) \cdot
        \begin{cases}
        \left(\frac{5N^{1/2}}{2\pi}(|t|+D_{s,f})\right)^{1.4-\sigma} &\text{if } |t|<5 \\
        \left(\frac{3N^{1/2}}{4\pi}(|t|+D_{s,f})\right)^{1.4-\sigma} &\text{if } |t|\geq 5,
        \end{cases}
    \end{equation}
    where
    $$
    Q(s)=N\frac{s+\varepsilon'+ir'}{2\pi}\cdot \frac{s+\varepsilon'-ir'}{2\pi},
    $$ 
   $\varepsilon' \in \{0,1\}$ is the parity of the cusp form, $1/4+r'^2$ is the Laplacian eigenvalue, $D_{s,f}:=3\sigma-1+\varepsilon'+|r'|+\frac{(2\sigma-1)^2}{1-\sigma+\varepsilon'}$ and $\zeta\left(1.4+\frac{7}{64}\right)\zeta\left(1.4-\frac{7}{64}\right)\approx 10.4$.
\end{example}

Since $|Q(2+s)|^{\frac{1+\varepsilon-\sigma}{2}} \asymp |t|^{1+\varepsilon-\sigma}$ as $|t|\to \infty$ for fixed $\sigma$ and $\varepsilon$, Example \ref{ex:newform} improves \cite[Corollary 4.3]{BT2018} for all fixed $\sigma\in [1/2,1)$, $\sigma >\varepsilon$, and $|t|$ large enough. The estimate \eqref{eq:Maass} improves \cite[Corollary 4.3]{BT2018} for all $\sigma \in [1/2,1)$ and $|t|\geq 5$. 

Similarly, Theorem \ref{thm:S} improves \cite[Theorem 7.1]{BT2018} for all fixed $\varepsilon \in (7/64, 1/2)$, and $t_1, t_2$ large enough. Note that due to Remark \ref{rmk:mistake}, the term $\frac{2}{\sqrt{2}(X-5)}$ in \cite[Theorem 7.1]{BT2018} should be replaced by $\frac{1.6}{X-5}$. 

In general, we obtain the following improvements.

\begin{itemize}
    \item Let $\varepsilon \in(0,1/2)$ and $\sigma \in [-\varepsilon, 1+\varepsilon]$ be fixed. The estimate for $|L(s)|$ by Lemma \ref{lemma:L} is $\asymp_\varepsilon |t|^{\frac{r(1+\varepsilon-\sigma)}{2}}$ as $|t| \to \infty$ whereas by Stirling's formula the estimate in \cite[Lemma 4.1]{Booker2006} is $\asymp |t|^{\frac{r(3-2\sigma)}{4}}$ as $|t| \to \infty$. Hence, Lemma \ref{lemma:L} improves \cite[Lemma 4.1]{Booker2006} for all $\sigma$ and for all $|t|$ large enough.

    Due to similar reasons, Theorem \ref{thm:S} improves \cite[Theorem 4.6]{Booker2006} for all fixed $\varepsilon \in (\theta,1/2)$, and $t_1, t_2$ large enough. 
   
    \item Let $L, \varepsilon$ and $\lambda_j$, $j=1,\ldots, f$, satisfy the same hypothesis as in Remark \ref{rmk:SelbergLs}. The exponent of the term $|t|$ in $Q_{L}(s)$ is $d_L$ where $d_L:=2\sum_{j=1}^f \lambda_j$ denotes the degree of the $L$-function. Thus for all $L$ and $\Re(s) \in [-\varepsilon, 1+\varepsilon]$ and $|t|$ large enough, Remark \ref{rmk:SelbergLs} improves the estimate for $|L(s)|$ in \cite[Theorem 4.2]{P2019}. The lower bound $t_0$ from which Remark \ref{rmk:SelbergLs} provides a better estimate than \cite[Theorem 4.2]{P2019} depends on exact parameters $f, \lambda_j, \mu_j$, $j=1,\ldots, f$.
\end{itemize}

\section{Estimating $L(s)$ when $-\varepsilon \leq \Re(s)\leq 1+\varepsilon$ and proving Lemma \ref{lemma:Sf}}

Our goal in this section is to obtain some explicit estimates on the $L$-functions described in Section \ref{sec:Artin} in the critical strip required for applying Turing's method. As seen in earlier sections, these bounds are also of independent interests. 

In some parts of this section, we require that $\lambda_j<1$ for all $j$ in the case of the Selberg class. Hence, we first prove Lemma \ref{lemma:Sf} to show that this assumption is not too restrictive.
\begin{proof}[Proof of Lemma \ref{lemma:Sf}]
Let $L$ be in the Selberg class, and let $N, \omega, \kappa, \lambda_j, \mu_j$ be as in the functional equation for that function. Assume that $\lambda_i\geq 1$ for some index $i \in \{1,2,\ldots,f\}$. By the Legendre duplication formula, we can write
\begin{equation*}
    \Gamma\left(\lambda_is+\mu_i\right)=\frac{2^{\lambda_is+\mu_i-1}}{\sqrt{\pi}}\Gamma\left(\frac{\lambda_is+\mu_i}{2}\right)\Gamma\left(\frac{\lambda_is+\mu_i+1}{2}\right),
\end{equation*}
and hence $L(s)$ satisfies the functional equation
$\mathcal{L}(s)=\omega_{\text{new}}\overline{\mathcal{L}(1-\bar{s})}$, where
    \begin{equation}
    \label{eq:newF}
    \mathcal{L}(s)=L(s)N_{\text{new}}^s\prod_{j=1}^{f_{\text{new}}}\Gamma\left(\lambda_{j,\text{new}} s+\mu_{j,\text{new}}\right), 
    \end{equation}
    \begin{align*}
        &\omega_{\text{new}}:=4^{-\Im(\mu_i)i}\omega, \text{ } f_{\text{new}}:=f+1, \text{ } N_{\text{new}}:=2^{\lambda_i}N, \text{ } \lambda_{f+1,\text{new}}=\lambda_{f+1,\text{new}}:=\lambda_i/2,  \\
        &\lambda_{j,\text{new}}:=\lambda_j, \text{ } \mu_{f+1,\text{new}}:=\mu_i/2+1/2, \text{ } \mu_{i,\text{new}}:=\mu_i/2 \text{ and } \mu_{j,\text{new}}:=\mu_j
    \end{align*}
     for $j=1,\ldots, f$, $j \neq i$. The completed $L$-function $\mathcal{L}$ in \eqref{eq:newF} also has form \eqref{eq:SelbergFunctionEq}. Repeat this process until all of the terms in $\lambda_{j,\text{new}}$ in \eqref{eq:newF} are $<1$, and the result is obtained.
\end{proof}

Now, we derive an estimate for gamma functions that has an important role in our estimates for the $L$-functions. The first one of the next results provides an estimate in that is used for the functions in the general set of $L$-functions, and the second one is used for the Selberg class functions. Let us begin with a variant of the Phragm\'{e}n\textendash Lindel\"{o}f theorem proved by Rademacher.

\begin{lemma}[{\cite[Theorem 2]{Rademacher1959}}]\label{lemma Rademacher}
    Let $g(s)$ be an analytic function of finite order in the strip $S(a,b)=\{s: a\leq \sigma\leq b\}$ such that
    \[
    \begin{cases}
        |g(a+it)|\leq A\prod_{j=1}^m |Q_j+a+it|^\alpha \\
        |g(b+it)|\leq B\prod_{j=1}^m |Q_j+b+it|^\beta
    \end{cases}
    \]
    with $\Re(Q_j)+a>0$ and $\alpha\geq \beta$. Then for any $s\in S(a,b)$,
    \[
    |g(s)|\leq \Bigg(A\prod_{j=1}^m|Q_j+s|^\alpha\Bigg)^{\frac{b-\sigma}{b-a}} \Bigg(B\prod_{j=1}^m|Q_j+s|^\beta\Bigg)^{\frac{\sigma-a}{b-a}}.
    \]
\end{lemma}

This allows us the prove the following:
\begin{lemma}
    \label{lemma:Gamma}
    Assume that $\mu_j \in \C$ with $\Re(\mu_j) \geq -1/2$. For $\varepsilon \in [-1/2,1/2]$,
    \[
    \bigg|\frac{\Gamma(\frac{\mu_j+1+\varepsilon+it}{2})}{\Gamma(\frac{\mu_j-\varepsilon+it}{2})}\bigg|\leq \bigg|\frac{\mu_j+2-\varepsilon+it}{2}\bigg|^{1/2+\varepsilon}. 
    \]
\end{lemma}

\begin{proof}
    We apply Lemma~\ref{lemma Rademacher} with $f(s)=\frac{\Gamma((\Re(\mu_j)+1-s)/2)}{\Gamma((\Re(\mu_j)+s)/2)}$ where $s=-\varepsilon-it-i\Im(\mu_j)$. Since
    \[\left|\Gamma\left(\frac{\mu_j-\varepsilon+it}{2}\right)\right|=\left|\Gamma\left(\frac{\overline{\mu_j}-\varepsilon-it}{2}\right)\right|,
    \]
    it suffices to seek a bound on $|f(s)|$. Take $a=-1/2, b=1/2$ and $Q=2+\Re(\mu_j)$, then we would have $A=1/2$, $B=1$, $\alpha=1$ and $\beta=0$. In particular,
    \[
    |f(a+it)|=\left|\frac{\Gamma(\frac{\Re(\mu_j)+3/2-it}{2})}{\Gamma(\frac{\Re(\mu_j)-1/2-it}{2})}\right| = \left|\frac{\Re(\mu_j)-1/2+it}{2}\right|\leq \left|\frac{\Re(\mu_j)+3/2+it}{2}\right|=\left|\frac{Q+a+it}{2}\right|
    \]
    and $|f(b+it)|=1$. It thus follows from Lemma~\ref{lemma Rademacher} that
    \[
    |f(s)|\leq \left|\frac{Q+s}{2}\right|^{1/2+\varepsilon},
    \]
    which yields the claim.
\end{proof}

\begin{remark}
    If $\Re(\mu_j)$ is very large and $\varepsilon<1/2$, one may want to use \cite[Theorem 2a]{Rademacher1959} instead. In this case we work with
    $f(s)=\frac{\Gamma((\Re(\mu_j)+s)/2)}{\Gamma((\Re(\mu_j)+1-s)/2)}$ where  $s=1+\varepsilon+it+i\Im(\mu_j)$, $a=1/2$, $b=3/2$, $A=1$, $B=1/2$, $\alpha=0$, $\beta=1$. We have $|f(a+it)|=1$ and
    $$
    |f(b+it)| \leq \left|\frac{\overline{\mu_j}-1/2-it}{2}\right| \leq \left|\frac{\mu_j+3/2+it}{2}\right|.
    $$
    Thus
    \[
    \bigg|\frac{\Gamma(\frac{\mu_j+1+\varepsilon+it}{2})}{\Gamma(\frac{\mu_j-\varepsilon+it}{2})}\bigg|\leq \left(\frac{\Re(\mu_j)+3/2}{\Re(\mu_j)+1/2}\right)^2\bigg|\frac{\mu_j+1+\varepsilon+it}{2}\bigg|^{1/2+\varepsilon}. 
\]
\end{remark}

\begin{remark}
\label{rmk:GammaSelberg}
     Assume that $\lambda_j \in \R_+$, $\lambda_j<1$, $\mu_j \in \C$ and $\Re(\mu_j)\geq 0$. Also let $\varepsilon \in [-1/2, (\lambda_j^{-1}-1)/2]$. We set $\Re(\mu_j)+1$
     in place of $Q$, $i\lambda_j t+i\Im(\mu_j)$ in place of $it$, $f(s):=\frac{\Gamma(\Re(\mu_j)+\lambda_j-s)}{\Gamma(\Re(\mu_j)+s))}$, $a:=(\lambda_j-1)/2$, $b:=\lambda_j/2$, $A=B=\alpha=1$ and $\beta=0$ in Lemma~\ref{lemma Rademacher}. Then  
    \[
\bigg|\frac{\Gamma(\mu_j+\lambda_j(1+\varepsilon)+i\lambda_j t)}{\Gamma(\mu_j-\lambda_j\varepsilon+i\lambda_j t)}\bigg|\leq \bigg|\mu_j+1+(-\varepsilon+it)\lambda_j\bigg|^{2\lambda_j(1/2+\varepsilon)}, 
\]
Notice that if $\lambda_j \leq 1/2$, then the exponent is $\le 1/2+\varepsilon$.
\end{remark}

The following estimate on $L(s)$ sharpens \cite[Lemma 4.1]{Booker2006}. In particular, when $\sigma=1/2$ this gives the so-called convexity bound on the critical line.

\begin{lemma}
\label{lemma:L}
    Let $L$ be as in Section \ref{sec:Artin}, $\varepsilon\in (0,1/2]$ and define $b_{1+\varepsilon}:=\sup_{\Re(s)=1+\varepsilon}|L(s)|$. Then for $s$ in the strip $\Re(s)\in [-\varepsilon,1+\varepsilon]$, we have
    \[
    |L(s)|\leq b_{1+\varepsilon}|Q(2+s)|^{\frac{1+\varepsilon-\sigma}{2}}\bigg|\frac{P(s-2)}{P(s-1)}\bigg|.
    \]
\end{lemma}

\begin{proof}
First consider the case where $L(s)$ is entire. Since $|L(s)|=|\chi(s)||L(1-\overline{s})|$, we can estimate
\begin{equation*}
    |L(-\varepsilon+it)|=|L(1+\varepsilon+it)| |\chi(-\varepsilon+it)|\leq b_{1+\varepsilon}\bigg|\frac{\gamma(1+\varepsilon+it)}{\gamma(-\varepsilon+it)}\bigg|=b_{1+\varepsilon} N^{1/2+\varepsilon}\pi^{-r/2-r\varepsilon} \prod_{j=1}^r \bigg|\frac{\Gamma(\frac{\mu_j+1+\varepsilon+it}{2})}{\Gamma(\frac{\mu_j-\varepsilon+it}{2})}\bigg|.
\end{equation*}
Invoking Lemma \ref{lemma:Gamma} in the penultimate line, the right-hand side is
\begin{align*}
    \leq b_{1+\varepsilon} N^{1/2+\varepsilon}\pi^{-r/2-r\varepsilon}\prod_{j=1}^r \bigg|\frac{\mu_j+2-\varepsilon+it}{2}\bigg|^{1/2+\varepsilon}= b_{1+\varepsilon} |Q(2-\varepsilon+it)|^{1/2+\varepsilon}.
\end{align*}

 Note that we may slightly modify \cite[Theorem 1]{Rademacher1959}, which is used in the proof of Lemma~\ref{lemma Rademacher}, by replacing the terms $Q+a+it$ and $Q+b+it$ in display equation (2.2) with $C\prod_j^r (Q_j+a+it)$ for some constant $C$ (resp. for $b$). Combining these with the bound $|L(1+\varepsilon+it)|\leq b_{1+\varepsilon}$ and applying Lemma~\ref{lemma Rademacher} gives the desired estimate.

Next suppose that $L(s)$ has poles. Note that $|s-1|\leq |s-2|$ throughout the half-plane $\Re(s)\le 3/2$, and hence $|P(s-1)|\leq |P(s-2)|$. The argument above works after we replace $L(s)$ by $L(s)P(s-1)/P(s-2)$, which is holomorphic in this half-plane. This completes the proof.
\end{proof}

\begin{remark}
\label{rmk:SelbergLs}
     Let $L$ be in the Selberg class with a functional equation where $\lambda_j<1$ for all j. Let $\varepsilon \in (0, \min\left\{\frac12,\frac{1}{2}(\max_j\{\lambda_j\}^{-1}-1)\right\}]$ and $Q_{L}(s)$ be as in \eqref{def:Qs}.
    Then for $\Re(s) \in [-\varepsilon, 1+\varepsilon]$ we have
     \[
    |L(s)|\leq b_{1+\varepsilon}|Q_{L}(s)|^{\frac{1+\varepsilon-\sigma}{2}}\bigg|\frac{s-2}{s-1}\bigg|^{k_L}.
    \]
\end{remark}

\section{Proofs of examples \ref{ex:Artin} and \ref{ex:newform}}\label{sec: proofs of examples}
In this section, we provide the proofs of the examples given in Section \ref{sec:examples}.

\begin{proof}[Proof of Example \ref{ex:Artin}]
      By \cite[Lemma 4]{GL2022}, we have $b_{1+\varepsilon} \leq \zeta(1+\varepsilon)^r$. Using $\varepsilon=0.49$ and Lemma \ref{lemma:L}, we obtain estimate \ref{eq:LArtin} 
    for $\Re(s) \in[0.5,1.49]$.
\end{proof}

\begin{proof}[Proof of Example \ref{ex:newform}]
 Note that $L(s)$ is an entire function. By Lemma \ref{lemma:L},
  \begin{equation*}
    |L(s)|\leq b_{1+\varepsilon}|Q(2+s)|^{\frac{1+\varepsilon-\sigma}{2}},  \text{ for } \varepsilon \in (0, 1/2], \text{ } \sigma \in [-\varepsilon, 1+\varepsilon],
  \end{equation*}
 To obtain the bound \eqref{eq:Maass}, we choose $\varepsilon=0.4$. Using \cite[Proposition 2]{KS2003}, we can estimate 
    \begin{equation*}
        b_{1+\varepsilon} \leq \left|\zeta\left(1.4+\frac{7}{64}\right)\zeta\left(1.4-\frac{7}{64}\right)\right|.
    \end{equation*}
    In addition, we have
    \begin{equation*}
        \left|Q(s+2)\right|^{1/2} \leq \frac{N^{1/2}}{2\pi}\left(|t|+2+\sigma+\varepsilon'+|r'|\right) \leq
        \begin{cases}
        \frac{5N^{1/2}}{2\pi}\left(|t|+D_{s,f}\right) & \text{if } |t| <5, \\
        \frac{3N^{1/2}}{4\pi}\left(|t|+D_{s,f}\right) & \text{if } |t| \geq 5,
        \end{cases}
    \end{equation*}
    where $D_{s,f}:=3\sigma-1+\varepsilon'+|r'|+\frac{(2\sigma-1)^2}{1-\sigma+\varepsilon'}$. Hence, the claim follows.
\end{proof}

\section{Estimating $\int S(t) \, dt$}
In this section, we provide the proofs of our main theorems described in Section \ref{sec:results}. First, we record two lemmas that will be useful for bounding the integral of the argument of $L(s)$ in those theorems.

\begin{lemma}
\label{lemma:gammaQ}
    \cite[Lemma 4.3(i), modified version]{Booker2006}
    Assume that $\varepsilon\in [0, 1/2]$, $\sigma \in [1/2, 2+\varepsilon]$ and that for all $j=1,...,r$, $(t+\Im(\mu_j))^2 \ge \left(2+\varepsilon+\Re(\mu_j)\right)^2+X^2$ for some $X>5$. Then
    \begin{equation*}
        -r\left(\frac{1}{2\sqrt{2}X}+\frac{4/\pi^2+1/4}{X^2}\right) \le \Re\frac{\gamma'}{\gamma}(\sigma+it)-\frac{1}{2}\log{\left|Q\left(1+\varepsilon+it\right)\right|} \leq \frac{4r}{\pi^2 X^2}.
    \end{equation*}
\end{lemma}

\begin{remark}
\label{rmk:gammaQdifference}
Let $L$ and $\lambda_j$ be as in Remark \ref{rmk:SelbergLs}. Assume also $\varepsilon \in [0, \min\left\{\frac12,\frac{1}{2}(\max_j\{\lambda_j\}^{-1}-1)\right\}]$, $\sigma \in [1/2, 2+\varepsilon]$ and for all $j=1,...,f$ and for some $X>5$ we have
$$
(\lambda_j t+\Im(\mu_j))^2 \ge \left(\lambda_j(2+\varepsilon)+\Re(\mu_j)\right)^2+X^2.
$$ 
Then, following the lines of the proof of \cite[Lemma 4.3(i)]{Booker2006}, we find that
\begin{equation*}
        -f\left(\frac{3}{4\sqrt{2}X}+\frac{4/\pi^2+1}{X^2}\right) \le \Re\frac{\gamma_{L}'}{\gamma_{L}}(\sigma+it)-\frac{1}{2}\log{\left|Q_{L}\left(1+\varepsilon+it\right)\right|} \leq \frac{4f}{\pi^2 X^2},
    \end{equation*}
    where $Q_{L}$ is as in \eqref{def:Qs} and $\gamma_L(s)$ as in \eqref{eq:SelbergFunctionEq}.
\end{remark}

\begin{lemma}
    \label{lemma:easyIntegral}
    \cite[Lemma 4.4, modified version]{Booker2006}
    Let $\kappa \in \C$ such that $|\Re(\kappa)|\le 1/2$ and $\varepsilon \in [0, 1/2]$. Then
    \begin{equation*}
        \int_0^{1/2+\varepsilon} \log{\left|\frac{(x+1+\kappa)(x+1-\overline{\kappa})}{(x+\kappa)(x-\overline{\kappa})}\right|} \, dx
        \leq A_\varepsilon\Re\left(\frac{1}{1+\kappa}+\frac{1}{1-\overline{\kappa}}\right).
    \end{equation*}
    where $A_\varepsilon$ is as in \eqref{def:Aepsilon}.
\end{lemma}

Let us now start with the case of a general set of functions containing entire Artin $L$-functions.

\begin{proof}[Proof of Theorem \ref{thm:S}]
Recall that (see \cite[Section 9.9]{Titc1986})
\[
\pi \int_{t_1}^{t_2} S(t)\,\md t=\int_{1/2}^\infty \log |L(\sigma+it_2)|\,\md \sigma - \int_{1/2}^\infty \log |L(\sigma+it_1)|\,\md \sigma.
\]
We simply derive an upper bound for the first integral and the lower bound for the second one.

We first address the upper bound for $t=t_2$. By Lemma \ref{lemma:L},
\begin{align*}
    \int_{1/2}^\infty \log |L(\sigma+it)|\,\md \sigma=& \int_{1/2}^{1+\varepsilon} \log |L(\sigma+it)|\,\md \sigma + \int_{1+\varepsilon}^\infty \log |L(\sigma+it)|\,\md \sigma\\
    \leq &\left(\frac{1}{2}+\varepsilon\right) \log b_{1+\varepsilon}+\int_{1/2}^{1+\varepsilon} \left(\frac{1+\varepsilon-\sigma}{2}\right)\log |Q(2+\sigma+it)|\,\md \sigma\\
    &\hspace{1cm}+\int_{1/2}^{1+\varepsilon}\log \bigg|\frac{P(\sigma-2+it)}{P(\sigma-1+it)}\bigg|\,\md \sigma+ \int_{1+\varepsilon}^\infty \log |L(\sigma+it)|\,\md \sigma.
\end{align*}
The first and last terms can be bounded by 
\[
\left(\frac{1}{2}+\varepsilon\right)r\log Z_\theta(1+\varepsilon)+r\int_{1+\varepsilon}^\infty \log Z_\theta(\sigma)\, \md \sigma
\]
(see \eqref{eq:cdef} for the definition of $Z_\theta$) using \cite[Lemma 4.5]{Booker2006}. The second integral is at most
\[
\int_{1/2}^{1+\varepsilon} \left(\frac{1+\varepsilon-\sigma}{2}\right)\,\md \sigma  \cdot \log |Q(3+\varepsilon+it)|= \left(\frac{1}{16}+\frac{\varepsilon(1+\varepsilon)}{4}\right) \log |Q(3+\varepsilon+it)|.
\]
For the third term, note that by the mean value theorem, for each $\sigma\in [1/2,1+\varepsilon]$ the integrand equals $\Re \frac{P'}{P}(\sigma^*+it)$ for some $\sigma^*\in [\sigma-2,\sigma-1]$, and hence the integral is at most 
\begin{equation}
\label{eq:PlogDer}
\left(\frac{1}{2}+\varepsilon\right) \max_{\sigma^* \in [-3/2, \varepsilon]}\left\{\sum_{k=1}^m \Re\frac{1}{\sigma^*+it-\tau_k}\right\}\leq \frac{\left(\frac{1}{2}+\varepsilon\right) \varepsilon m}{\varepsilon^2+\min_{1\leq j\leq r}(t+\Im \mu_j)^2} \leq \frac{\left(\frac{1}{2}+\varepsilon\right) \varepsilon m}{X^2}.
\end{equation}

Now we turn to the lower bound for $t=t_1$. We can write
\begin{align*}
    \int_{1/2}^\infty \log{\left|L(\sigma+it)\right|} \, d\sigma=&\int_{1/2}^\infty \log{\left|\frac{L(\sigma+it)}{L(\sigma+1+it)}\right|} \, d\sigma+\int_{3/2}^\infty \log{\left|L(\sigma+it)\right|} \, d\sigma\\
    =&\int_{1/2}^{1+\varepsilon} \log{\left|\frac{F(\sigma+it)}{F(\sigma+1+it)}\right|} \, d\sigma+\int_{1/2}^{1+\varepsilon} \log{\left|\frac{\gamma(\sigma+1+it)}{\gamma(\sigma+it)}\right|} \, d\sigma \\
    &\hspace{1cm}+\int_{1/2}^{1+\varepsilon} \log{\left|\frac{P(\sigma+1+it)}{P(\sigma-1+it)}\right|} \, d\sigma+\int_{1+\varepsilon}^{2+\varepsilon} \log{|L(\sigma+it)|} \, d\sigma\\
    &\hspace{2cm}+\int_{3/2}^\infty\log{|L(\sigma+it)|} \, d\sigma.
\end{align*}
The claim follows similarly as in the proof of \cite[Theorem 4.6]{Booker2006}, but we apply modified versions of \cite[Lemma 4.3(i) \& 4.4]{Booker2006} (see Lemmas \ref{lemma:gammaQ} and \ref{lemma:easyIntegral}, respectively) and the same idea as in \eqref{eq:PlogDer} for the logarithmic integral of $P$. In addition, we estimate
\begin{multline*}
    r\left(\frac{1}{2\sqrt{2}X}+\frac{4\pi^2+1/4}{X^2}+\frac{4A_\varepsilon}{\pi^2 X^2}\right)+\frac{(1/2+\varepsilon)(2+\varepsilon)m}{X^2} 
    \leq \frac{r\left(8.3+0.09A_\varepsilon+(1/2+\varepsilon)(2+\varepsilon)/5\right)}{X},
\end{multline*}
where we have used the facts that $r\geq m$ and $X>5$. 
\end{proof}

\begin{remark}
\label{rmk:mistake}
    We note that there is a small mistake in \cite[Theorem 4.6]{Booker2006}. The last inequality of the proof holds only if $5<X<10.5998$, not for all $X>5$ as assumed in \cite[Theorem 4.6]{Booker2006}. However, the theorem is valid if we replace $\frac{r}{\sqrt{2}(X-5)}$ by $\frac{0.8r}{X-5}$.
\end{remark}

The following two proofs consider the case of the Selberg class.

\begin{proof}[Proof of Theorem \ref{thm:SelbergL}]
    We approach as in the proof of Theorem \ref{thm:S}. However, since we are considering a set of functions that satisfy slightly different assumptions, we make the changes detailed in Table~\ref{table:SelbergChanges}. In the case where we estimate the integral 
    $$
        \int_{1/2}^{1+\varepsilon} \frac{(\sigma+it)^{k_L}(\sigma+it-1)^{k_L}\mathcal{L}(\sigma+it)}{(\sigma+1+it)^{k_L}(\sigma+it)^{k_L}\mathcal{L}(\sigma+1+it)} \, dt,
    $$
    we use the fact that (see \cite[Lemma 3.2]{PS2024})
    $$
        s^{k_L}(s-1)^{k_L}\mathcal{L}(s)=e^{A_L s+B_L}\prod_\rho \left(1-\frac{s}{\rho}\right)e^{s/\rho},
    $$
    where $A_L$ and $B_L$ depend on $L$ and $\Re(B_L)=-\sum_\rho \Re(1/\rho)$. Notice that $t_2 \geq \varepsilon$, $t_2\geq t_1$ and hence
    \begin{multline*}
        \frac{\left(\frac{1}{2}+\varepsilon\right)\varepsilon k_L}{t_2^2}+f\left(\frac{3}{4\sqrt{2}X}+\frac{4/\pi^2+1}{X^2}\right)+\frac{4fA_\varepsilon}{\pi^2 X^2}+\frac{2k_L}{\max\{1/2,t_1\}} \\
        \leq \frac{f}{X}\left(0.9+\frac{4A_\varepsilon}{5}\right)+\frac{k_L(2.5+\varepsilon)}{\max\{\varepsilon,t_1\}}.
    \end{multline*}
    The results follows.
\end{proof}

\begin{proof}[Proof of Theorem \ref{thm:SelbergPolynomial}]
    We can follow the same lines as in the proof of Theorem \ref{thm:SelbergL}. However, we can benefit from the polynomial Euler product representation by replacing $Z_{\theta_{L}}(\sigma)$, $z_{\theta_{L}}(\sigma)$ with $Z_0(\sigma)^m$ and $z_0(\sigma)^m$, respectively. Note that in this case we also have $\theta_{L}=0$.
\end{proof}

\begin{table}
\begin{center}
\begin{tabular}{ |c|c| } 
 \hline
 \textbf{Proof of Theorem \ref{thm:S}} & \textbf{The Selberg class}  \\ \hline
 $Q(2+\sigma+it)$ & $Q_{L}(\sigma+it)$   \\ \hline
 $Q(3+\varepsilon+it_2)$ & $Q_{L}(1+\varepsilon+it_2)$   \\ \hline
 $Q(1+\varepsilon+it_1)$ & $Q_{L}(1+\varepsilon+it_1)$   \\ \hline
 $P(s)$ & $s^{k_L}$  \\ \hline
 $r$ & $f$  \\ \hline
 $Z_{\theta}(\sigma)^r$& $Z_{\theta_{L}}(\sigma)$  \\ \hline
 $z_{\theta}(\sigma)^r$& $z_{\theta_{L}}(\sigma)$  \\ \hline
 $\int_{1/2}^{1+\varepsilon}\log \bigg|\frac{P(\sigma-2+it)}{P(\sigma-1+it)}\bigg|\,\md \sigma \le \frac{\left(\frac{1}{2}+\varepsilon\right) \varepsilon m}{X^2}$ & $\ldots\le \frac{k_L\left(\frac{1}{2}+\varepsilon\right)\varepsilon}{t^2}$ \\ \hline
 $F(s)$& $s^{k_L}(s-1)^{k_L}\mathcal{L}(s)$ \\ \hline
 $\gamma$ & $\gamma_{L}$ \\ \hline
 Lemma \ref{lemma:gammaQ}& Remark \ref{rmk:gammaQdifference} \\ \hline
 Assume $|t_2+\Im(\mu_j)| \geq \frac{3}{2}$, $\forall j$ & Assume $t\ge \varepsilon$ \\ \hline
\end{tabular}
\caption{Changes in the proof of Theorem \ref{thm:S} when the Selberg class is considered.} 
\label{table:SelbergChanges}
\end{center}
\end{table}

\section{Rigorous computation of the Selberg class functions}
\label{sec:rigorous}
As seen in Section \ref{sec:Turing}, the efficient computations in Turing's method rely on fast computations of the function $\Lambda(s)$ given in \eqref{def:Lambdas}. We remember that we do not know whether Selberg class is a subset of the general class defined in Section \ref{sec:Artin} or not. Hence, in this section, we discuss how to apply the method to the Selberg class under some additional assumptions. Since most of the proofs are very similar to Booker's proofs in \cite[Section 5]{Booker2006} (see also Odlyzko and Sch\"onhage \cite{OS1988}), we do not give full detailed proofs for our statements. Instead, we clarify the main steps or differences to the Booker's ones and explain the general idea how the results can be applied. Throughout the section we assume that $L\in \cS$. The idea is that instead of determining the function $\Lambda(1/2+it)$ directly, we consider equivalently an inverse Fourier transform of a certain function that is easier to be estimated.

We will start with some useful definitions. Let $\omega$ be as in \eqref{eq:SelbergFunctionEq}, $z$ be a complex number such that
\begin{equation*}
    \Re(z):=\frac{\Im(\omega)}{\sqrt{2(1-\Re(\omega))}}, \quad \Im(z):=-\frac{\sqrt{1-\Re(\omega)}}{\sqrt{2}},
\end{equation*}
and 
\begin{equation}
\label{def:LambdaL}
\Lambda_{L}(s):=z\mathcal{L}(s).
\end{equation}
Now $|z|=1$ and $\Lambda_{L}(s)=\overline{\Lambda_{L}(1-\overline{s})}$, and thus $\Lambda_{L}(1/2+it) \in \mathbb{R}$. 
For some $\eta\in(-1,1)$ put $F(t):=\Lambda_{L}(1/2+it)e^{\frac{\pi f }{4}\eta t}$ and
\begin{equation}
\label{eq:defG}
    G(u; \eta, \{\lambda_j\}, \{\mu_j\}):=\sum_{\rho \in \mathbb{C}} \text{Res}_{s=\rho} \left(e^{(u+i\frac{\pi f}{4}\eta)(1/2-s)}\prod_{j=1}^f  \Gamma(\lambda_j s+\mu_j) \right),
\end{equation} 
where $f, \lambda_j$ and $\mu_j$ are as in \eqref{eq:SelbergFunctionEq}. The (inverse) Fourier transform is defined as 
$$
\hat{F}(x):=\frac{1}{2\pi}\int_{-\infty}^\infty F(t) e^{-ixt} \, dt=\frac{1}{2\pi}\int_{-\infty}^\infty \Lambda_{L}(1/2+it) e^{t\left(\frac{\pi f }{4}\eta -ix\right)} \, dt.
$$
Hence, we consider the function $\hat{F}$ instead of function $\Lambda_{L}$, and we want to determine it to any given precision.

First, we give an alternative interpretation for $F$.
\begin{lemma}
\label{lemma:F}
    Remember the notation in \eqref{eq:DS} and \eqref{eq:SelbergFunctionEq}. We have
    \begin{equation}
    \label{eq:FSum}
        \hat{F}(x)=z\sum_{n=1}^\infty \frac{a(n)}{\sqrt{n/N}}G\left(x+\log{\frac{n}{N}}; \eta, \{\lambda_j\}, \{\mu_j\}\right)-\text{Res}_{s=1}\Lambda_{L}(s)e^{(x+i\frac{\pi f}{4}\eta)(1/2-s)}.
    \end{equation}
\end{lemma}

\begin{proof}
    Similarly as in \cite[Equation (5--2)]{Booker2006} we can deduce that
     \begin{equation*}
        \hat{F}(x)=\frac{1}{2\pi i}\int_{\Re(s)=2}\Lambda_{L}(s)e^{(x+i\frac{\pi f}{4}\eta)(1/2-s)} \, ds -\text{Res}_{s=1}\Lambda_{L}(s)e^{(x+i\frac{\pi f}{4}\eta)(1/2-s)}.
    \end{equation*}
    Using the definitions of $\Lambda_{L}, \gamma_{L}$ (see \eqref{def:LambdaL} \& \eqref{eq:SelbergFunctionEq}) and Equation (5--3) in \cite{Booker2006}, the claim follows.
\end{proof}

\begin{remark}
    As we see from the proof of Lemma \ref{lemma:F}, we can write
    \begin{equation}
    \label{eq:rmkGdef}
       G(u; \eta, \{\lambda_j\}, \{\mu_j\}) =\frac{1}{2\pi i}\int_{\Re(s)=a}e^{(u+i\frac{\pi f}{4}\eta)(1/2-s)}\prod_{j=1}^f \Gamma(\lambda_js+\mu_j) \, ds
    \end{equation}
    for any $a>1$.
\end{remark}

The previous lemma tells that we would like to estimate function $G$ and 
\begin{equation}
\label{eq:residue}
    \text{Res}_{s=1}\Lambda_{L}(s)e^{(x+i\frac{\pi f}{4}\eta)(1/2-s)}
\end{equation} 
with any wanted precision. When we have fixed the $L$-function under consideration, the term in \eqref{eq:residue} is easy to compute by known theorems for residues.  
Hence, we concentrate on estimating the function $G$. The idea is that we can divide the sum in the definition of $G$ in \eqref{eq:defG} into different parts, noting that if $\rho$ is a pole of the function, then so is $\rho-l/\lambda_j$ for any non-negative integer $l$. Choosing the pole $\rho$ of $\Gamma(\lambda_j s+\mu_j)$ with the smallest absolute value and summing through all $l \geq 1$, and doing this for all $j=1,2, \ldots, f$, we have only finitely many cases to consider. The next lemma considers the sum over $l \geq 1$.

\begin{lemma}
\label{lemma:tails}
Let $\lambda_1=\lambda_2=\ldots=\lambda_f<1$ and let $\rho$ be a pole of 
\begin{equation*}
    g(s):= e^{\left(u+i\frac{\pi f}{4}\eta\right)(1/2-s)}\prod_{j=1}^{f}\Gamma(\lambda_j s+\mu_j)
\end{equation*}
of order $n$, with $\Re(\lambda_j\rho+\mu_j) \leq 0$ for $j=1,\ldots, f$ and 
\begin{equation}
\label{eq:12}
    \frac{e^{u/\lambda_1}}{\prod_{j=1}^f \left(|1-\lambda_j\rho-\mu_j|-\lambda_j\right)}<\frac{1}{2}.
\end{equation}
Let $c_j$ be the coefficients of the polar part of $g$ around $\rho$, so that
$g(s+\rho)-\sum_{j=1}^n c_js^{-j}$ is holomorphic at $s=0$. Then
\begin{equation*}
    \left|\sum_{l=1}^\infty \text{Res}_{s=\rho-l/\lambda_1} g(s)\right|<\max|c_j|.
\end{equation*}
\end{lemma}

\begin{proof}
    By the definition of the function $g(s)$ and the properties of the gamma function, we have
    \begin{multline}
    \label{eq:gExpression}
        g(s+\rho-1/\lambda_1)=e^{\left(u+i\frac{\pi f}{4}\eta\right)/\lambda_1}g(s+\rho)\prod_{j=1}^f \left(\lambda_j(s+\rho)-1+\mu_j\right)^{-1} \\
        =\frac{(-1)^f e^{\left(u+i\frac{\pi f}{4}\eta\right)/\lambda_1}}{\prod_{j=1}^f (1-\lambda_j\rho-\mu_j)}\cdot\frac{g(s+\rho)}{\prod_{j=1}^f \left(1-\frac{\lambda_j s}{1-\lambda_j\rho-\mu_j}\right)}.
    \end{multline}
    For all meromorphic functions $f(s)$ with polar part $a_1s^{-1}+\ldots+a_ns^{-n}$ at $0$ and complex numbers $|x|<1$, the function $f(s)/(1-xs)$ has polar $a_1's^{-1}+\ldots+a_n's^{-n}$, where $a_j'=\sum_{k=0}^{n-j} a_{j+k}x^k$ and
    \begin{equation*}
        \max |a_j'| \leq \frac{\max |a_j|}{1-|x|}.
    \end{equation*}
    Hence, using \eqref{eq:gExpression}, we have
    \begin{equation*}
        \max |c_j'| \leq \frac{e^{u/\lambda_1}}{\prod_{j=1}^f |1-\lambda_j\rho-\mu_j|}\cdot\frac{\max |c_j|}{\prod_{j=1}^f \left(1-\frac{\lambda_j}{\left|1-\lambda_j\rho-\mu_j\right|}\right)}=\frac{e^{u/\lambda_1} \max|c_j|}{\prod_{j=1}^f \left(|1-\lambda_j\rho-\mu_j|-\lambda_j\right)}<\frac{1}{2}\max |c_j|.
 \end{equation*}
    We can repeat this process for all $l=1,2,\ldots$ and conclude that the coefficients of the polar part of $g(s+\rho-l/\lambda_1)$ are less than $2^{-l}\max |c_j|$. Hence, the claim follows.
\end{proof}

The previous lemma tells us that we can determine the value of $G(u;\eta, \{\lambda_j\}, \{\mu_j\})$ for each $u$ by computing the sum in \eqref{eq:defG} in finitely many points $\rho$ (and this number depends on $u$) and then estimating the tails by Lemma \ref{lemma:tails}. The expression \eqref{eq:gExpression} gives a way to determine the data at $\rho-1/\lambda_1$ from the data at $\rho$, and shows that the terms are eventually decreasing. Hence, for a given $u$, the term $G(u;\eta, \{\lambda_j\}, \{\mu_j\})$ can be computed to arbitrary precision. A more detailed algorithm for the computations is given in \cite[p. 400]{Booker2006}. 

Let us now consider the cases where $u$ is large in $G(u;\eta, \{\lambda_j\}, \{\mu_j\})$.

\begin{lemma}
\label{lemma:uLarge}
Assume $\lambda_1=\lambda_2=\ldots=\lambda_f \geq 1/2$. Let 
\begin{align*}
    & \delta:= \frac{\pi}{2}\left(1-\frac{|\eta|}{2\lambda_1}\right), \quad \mu:=\frac{\lambda_1-1}{2\lambda_1}+\frac{1}{f\lambda_1}\left(\frac{1}{2}+\sum_{j=1}^f \mu_j\right), \quad \nu_j :=\Re(\mu_j)+\frac{1}{2}\left(\frac{1}{f}-1\right), \\
    & K:=\frac{1}{\pi}\sqrt{\frac{2^{f+1}}{f}\frac{e^{\delta(f-1)}}{\delta}}e^{-\frac{\pi f \eta \Im(\mu)}{4}} \quad \text{and}\quad X_1(u):=f\delta e^{-\delta+u/(f\lambda_1)}.
\end{align*} 
If $X_1(u)\geq f$, then
    \begin{equation*}
        \left|G(u; \eta, \{\lambda_j\}, \{\mu_j\})\right| \leq K e^{\Re(\mu)u}e^{-X_1(u)}\prod_{j=1}^f \left(1+\frac{f\nu_j}{X_1(u)}\right)^{\nu_j}.
    \end{equation*}
\end{lemma}

\begin{proof}
Setting $s=(\sigma+it)/\lambda_1$ in \eqref{eq:rmkGdef}, we obtain
\begin{equation*}
    \left|G(u; \eta, \{\lambda_j\}, \{\mu_j\})\right| \leq \frac{e^{u(1/2-\sigma/\lambda_1)-\pi \eta \Im(\mu)/4}}{2 \pi}\int_{-\infty}^\infty \prod_{j=1}^f \left|\Gamma(\sigma+it+\mu_j)e^{\frac{\pi\eta}{4\lambda_1}(t+\Im(\mu_j))}\right|.
\end{equation*}
By H\"older's inequality, it is sufficient to estimate integrals
\begin{equation*}
\int_{-\infty}^\infty \left|\Gamma(a_j+it)\right|^f e^{\frac{\pi f \eta}{4\lambda_1}t} \, dt,
\end{equation*}
where $a_j:= \sigma+\Re(\mu_j)$. The claim follows by the last paragraph of the proof of \cite[Lemma 5.2]{Booker2006}.
\end{proof}

Directly using the previous lemma and estimating a sum with an integral, we are able to estimate the first sum in \eqref{eq:FSum} when the index $n$ is large enough.
\begin{corollary}
\label{lemma:GInfinity}
Assume the same hypothesis as in Lemma \ref{lemma:uLarge}, and let $M$ be a positive integer and $x \in \mathbb{R}$. Let $\delta, \nu_j, \mu, K$ be as in Lemma \ref{lemma:uLarge}, and set $X_2(x):=X_1(x-\log{N})$. Let $C, \alpha$ be such that $|a(n)| \leq Cn^{\alpha}$ for all $n \geq 1$. Further set $c:=\Re(\mu)+1/2+\alpha$, $c':=\max\{cf\lambda_1-1,0\}$. Then for $X_2(x)M^{1/(f \lambda_1)}>\max\{c',f\}$ we have
\begin{align*}
    & \left|\sum_{n>M} \frac{a(n)}{\sqrt{n/N}}G\left(x+\log{\frac{n}{N}}; \eta, \{\lambda_j\}, \{\mu_j\}\right)\right| \\
    &\quad \leq \sqrt{N}K\lambda_1 f \left(\frac{e^x}{N}\right)^{\Re(\mu)}\frac{CM^c e^{-X_2(x)M^{1/(\lambda_1 f)}}}{X_2(x)M^{1/(\lambda_1 f)}-c'} \prod_{j=1}^{f}\left(1+\frac{f\nu_j}{X_2(x)M^{1/(\lambda_1 f)}}\right)^{\nu_j}.
\end{align*}
\end{corollary}

Combining Lemma \ref{lemma:tails} and Corollary \ref{lemma:GInfinity}, we can determine the sum over function $G$ to any given precision. Hence, the last step is to use the fast Fourier transform to compute $F$ from $\hat{F}$. This requires discretizing the problem. We can discretize the problem in the similar way as in \cite[p. 398]{Booker2006}. Namely, let $A,B>0$ be parameters such that $q=AB$ is an integer. Now, the functions
\begin{equation*}
    \widetilde{F}(m):=\sum_{l \in \Z} F\left(\frac{m}{A}+lB\right) \quad\text{and}\quad \widetilde{\hat{F}}(n):=\sum_{l \in \Z} \hat{F}\left(\frac{2\pi n}{B}+2\pi Al\right)
\end{equation*}
form a discrete Fourier transform pair. Moreover, the aforementioned functions are periodic in $m$ and $n$ with period $q$. Since $F$ is real-valued, we have $\hat{F}(-x)=\overline{\hat{F}(x)}$, and hence for $|n|\leq q/2$ we have
\begin{equation}
\label{eq:Ftildehat}
    \widetilde{\hat{F}}(n)=\hat{F}\left(\frac{2\pi n}{B}\right)+\sum_{l=1}^\infty\left( \hat{F}\left(\frac{2\pi n}{B}+2\pi l A\right)+\overline{\hat{F}\left(-\frac{2\pi n}{B}+2\pi l A\right)}\right).
\end{equation}
We already know a way to compute $\hat{F}$ in \eqref{eq:Ftildehat}. Hence, we are left with the sums over $l$, also in the case where we have $F$ instead of $\hat{F}$.

Let us consider the choice of parameters before estimating $\widetilde{\hat{F}}(n)$. By Stirling's formula, the function $F(t)$ decays roughly like $e^{-(1-\eta)\pi f \lambda_1 t/2}$ for $t>0$. Hence, if we would like to find the zeros up to height $T$, the term $1-\eta$ should be of size $T^{-1}$. However, the exact value of $\eta$ can be changed depending on how far we are from $T$ to get the best precision. Choosing $1-\eta \asymp T^{-1}$ means also that $\delta \asymp T^{-1}$. Moreover, we will also choose $B$ to be a multiple of $T$ where the exact multiple depends on the chosen $\nu$. To consider that, we note that the zero-density of a Selberg class function $L$ at height $T$ is roughly $\log\left((T/e)^{d_L}\lambda N^2\right)/(2\pi)$ (see \cite{Steuding2003}), where $\lambda$ is as in \eqref{def:lambda}. We want $A$ to be a multiple of this.

Let us start with the sums over $l$ in the case of $\hat{F}$ when $2\pi n/B$ is large enough. 

\begin{lemma}
Assume $\lambda_1=\lambda_2=\ldots=\lambda_f \geq 1/2$. Let  $\delta, \nu_j, \mu, K$ be as in Lemma \ref{lemma:uLarge} and $X_2, C, c'$ as in Corollary \ref{lemma:GInfinity}. Moreover, let $x\in \R$ such that $X_2(x) >\max\{c',f\}$ and $A\geq 1/(2 \pi)$. Then
\begin{align*}
    &\left|\sum_{l=0}^\infty \hat{F}\left(x+2\pi l A\right)+\text{Res}_{s=1}\left(\frac{ \Lambda_{L}(s)e^{(x+i\frac{\pi r \eta}{4})(1/2-s)}}{1-e^{2\pi A(1/2-s)}}\right)\right|\leq \frac{ z\sqrt{N}K}{1-e^{-\pi A}}\left(\frac{e^x}{N}\right)^{\Re(\mu)}e^{-X_2(x)} \\
    &\quad\times \left(1+\frac{\lambda_1 f C}{X_2(x)-c'}\right)\prod_{j=1}^f\left(1+\frac{f \nu_j}{X_2(x)}\right)^{\nu_j}.
\end{align*}
\end{lemma}

\begin{proof}
    The residue part comes from the last term in \eqref{eq:FSum} and computing the geometric sum 
    $$
    \sum_{l=0}^\infty \left(e^{2 \pi A \left(1/2-s\right)}\right)^l=\frac{1}{1-e^{2 \pi A (1/2-s)}}.
    $$
    For the rest of the term and $l=0$, by Lemma \ref{lemma:uLarge} and Corollary \ref{lemma:GInfinity} with $M=1$ we have the bound
    \begin{equation*}
        z\sqrt{N}K \left(\frac{e^x}{N}\right)^{\Re(\mu)}e^{-X_2(x)}\left(1+\frac{\lambda_1fC}{X_2(x)-c'}\right)\prod_{j=1}^f \left(1+\frac{f \nu_j}{X_2(x)}\right)^{\nu_j}.
    \end{equation*}  
     Since $\Re(\mu)=c-1/\alpha$, $c \leq X_2(x)/(\lambda_1 f)$, $e^y-y-1 <0$ for all $y>0$ and $A \geq 1/(2 \pi)$, we have
    \begin{equation*}
        e^{2\pi lA\Re(\mu)}e^{-X_2(x)\left(e^{2\pi l A/(f \lambda_1)-1}\right)} \leq \exp\left(-2 \pi l A\left(\frac{1}{2}+\alpha+\frac{2\pi l A-1}{\lambda_1 f}\right)\right) \leq e^{-\pi l A}.
    \end{equation*}
    Noticing that $X_2(y)$ is an increasing function of $y$, we have a geometric sum over $l$ left and the claim follows. 
\end{proof}

The final corollary considers the error coming from the terms $F$ in \ref{eq:Ftildehat}. It is a consequence of Remark \ref{rmk:SelbergLs} and identities described in Table \ref{table:SelbergChanges}.

\begin{corollary}
Assume the same hypothesis as in Remark \ref{rmk:SelbergLs} and that $\varepsilon>\theta_{L}$, where $\theta_{L}$ is as in the Euler product representation of the Selberg class functions. Let $t \in \R$, $s:=1/2+it$,
\begin{equation*}
    E_{\theta_{L}}(\varepsilon):=Z_{\theta_{L}}(1+\varepsilon)\left|Q_{L}(s)\right|^{\frac{1/2+\varepsilon}{2}}\left|\gamma_{L}(s)\right|\left|\frac{s-2}{s-1}\right|^{k_L}e^{\frac{\pi f}{4}\eta t}
\end{equation*}
and
\begin{equation*}
    \beta_{L}:= \frac{d_L \pi}{4}-\sum_{j=1}^f \lambda_j\left(\arctan{\left(\frac{\Re(\lambda_j s+\mu_j)}{\left|\Im(\lambda_js+\mu_j)\right|}\right)}+\frac{1}{2\Re(\lambda_js+\mu_j)}+\frac{2}{\pi^2\left|\Im(\lambda_js+\mu_j)^2-\Re(\lambda_js+\mu_j)^2\right|}\right).
\end{equation*}
\begin{enumerate}[(i)]
\item If $\Im(\lambda_js+\mu_j)>0$ for all $j=1,\ldots,f$ and $\beta_{L}-\pi f \eta/4>0$, then
\begin{equation*}
\left|\sum_{l=0}^\infty F(t+lB)\right| \leq \frac{E_{\theta_{L}}(\varepsilon)}{1-e^{-(\beta_{L}-\pi f\eta/4)B}}.
\end{equation*}
\item  If $\Im(\lambda_js+\mu_j)<0$ for all $j=1,\ldots,f$ and $\beta_{L}+\pi f \eta/4>0$, then
\begin{equation*}
\left|\sum_{l=0}^\infty F(t-lB)\right| \leq \frac{E_{\theta_{L}}(\varepsilon)}{1-e^{-(\beta_{L}+\pi f\eta/4)B}}.
\end{equation*}
\end{enumerate}
\end{corollary}

\begin{remark}
    Using Lemma \ref{lemma:L}, we can replace $E$ in \cite[Lemma 5.7]{Booker2006} with
    \begin{equation*}
        E_{\theta}(\varepsilon):=Z_{\theta}(1+\varepsilon)\left|Q(2+s)\right|^{\frac{1/2+\varepsilon}{2}}\left|\gamma(s)\right|\left|\frac{P(s-2)}{P(s-1)}\right|,
    \end{equation*}
    where $\varepsilon\in (\theta,1/2]$, for $L$-functions in the set described in Section \ref{sec:Artin}.
\end{remark}

Indeed, the last too lemmas explain how it is sufficient to consider only finitely many terms in the sums in \eqref{eq:Ftildehat}. Using our results for the function $G$, those finitely many cases can be computed in any wanted precision, which was the wanted goal.

\section{Acknowledgment}
We would like to thank Ghaith Hiary for suggesting the research problem, as well as Siva Sankar Nair and Jesse J\"a\"asaari for discussions improving the manuscript. In addition, we would like to thank the organisers, Alia Hamieh, Ghaith Hiary, Habiba Kadiri, Allysa Lumley, Greg Martin and Nathan Ng, of the summer school ``Inclusive Paths in Explicit Number Theory,'' which was hosted by BIRS at UBC Okanagan during July 2023. This project would not have been possible without this amazing meeting. 

\section{Declaration of Funding}
N. Paloj\"arvi was supported by the Finnish Cultural Foundation.

\section{Declaration of Competing Interests}
The authors have no competing interests to declare.

\printbibliography

\end{document}